\documentclass[11pt]{article}
\makeatletter

\usepackage{amssymb,amsmath,amsthm,latexsym,color,verbatim}
\title{\bf Maximum scattered linear sets and complete caps in Galois spaces\thanks{\textcolor{black}{The
research  was supported by
Ministry for Education, University and Research of Italy MIUR (Project
PRIN 2012 "Geometrie di Galois e strutture di incidenza") and by the Italian National
Group for Algebraic and Geometric Structures and their Applications (GNSAGA
- INdAM)}}}
\author{Daniele Bartoli, Massimo Giulietti, Giuseppe Marino and Olga Polverino}\date{}

\providecommand{\keywords}[1]{\textbf{{Keywords:}} #1.}

\begin{document}
\maketitle

\newtheorem{theorem}{Theorem}[section]
\newtheorem{lemma}[theorem]{Lemma}
\newtheorem{remark}[theorem]{Remark}
\newtheorem{cor}[theorem]{Corollary}
\newtheorem{prop}[theorem]{Proposition}
\newtheorem{defin}[theorem]{Definition}

\makeatother
\newcommand{\Prf}{\noindent{\bf Proof}.\quad }
\renewcommand{\labelenumi}{(\alph{enumi})}
\renewcommand{\qed}{\hfill \mbox{\raggedright \rule{0.1in}{0.1in}}}


\def\B{\mathbf B}
\def\C{\mathbf C}
\def\Z{\mathbf Z}
\def\Q{\mathbf Q}
\def\W{\mathbf W}
\def\a{\mathbf a}
\def\b{\mathbf b}
\def\c{\mathbf c}
\def\d{\mathbf d}
\def\e{\mathbf e}
\def\l{\mathbf l}
\def\v{\mathbf v}
\def\w{\mathbf w}
\def\x{\mathbf x}
\def\y{\mathbf y}
\def\z{\mathbf z}
\def\t{\mathbf t}
\def\cD{\mathcal D}
\def\cC{\mathcal C}
\def\cH{\mathcal H}
\def\cM{{\mathcal M}}
\def\cK{\mathcal K}
\def\cQ{\mathcal Q}
\def\cU{\mathcal U}
\def\cS{\mathcal S}
\def\cT{\mathcal T}
\def\cR{\mathcal R}
\def\cN{\mathcal N}
\def\cA{\mathcal A}
\def\cF{\mathcal F}
\def\cL{\mathcal L}

\def\PG{{\rm PG}}
\def\GF{{\rm GF}}

\def\Pg{PG(5,q)}
\def\pg{PG(3,q^2)}
\def\ppg{PG(3,q)}
\def\HH{{\cal H}(2,q^2)}
\def\F{\mathbb F}
\def\Fn{\mathbb F_{q^n}}
\def\P{\mathbb P}
\def\V{\mathbb V}

\def\ps@headings{
 \def\@oddhead{\footnotesize\rm\hfill\runningheadodd\hfill\thepage}
 \def\@evenhead{\footnotesize\rm\thepage\hfill\runningheadeven\hfill}
 \def\@oddfoot{}
 \def\@evenfoot{\@oddfoot}
}

\begin{abstract}
\textcolor{black}{Explicit constructions of infinite families of scattered $\F_q$--linear sets in $PG(r-1,q^t)$ of maximal rank $\frac{rt}2$, for $t$ even, are provided. When $q=2$ and $r$ is odd, these linear sets correspond to complete caps in $AG(r,2^t)$ fixed by a translation group of size $2^{\frac{rt}2}$. The doubling construction applied to such caps  gives complete caps in $AG(r+1,2^t)$ of size $2^{\frac{rt}2+1}$. For Galois spaces of even dimension greater than $2$ and even square order, this solves the long-standing problem of establishing whether the theoretical lower bound for the size of a complete cap is substantially sharp.}
\end{abstract}


\bigskip

\par\noindent

\section{Introduction}

\medskip

Let $\Lambda=PG(V,\F_{q^t})=PG(r-1,q^t)$, $q=p^h$, $p$ prime, with $V$ vector space of dimension r over $\F_{q^t}$, and let
$L$ be a set of points of $\Lambda$. The set $L$ is said to be an
{\it $\F_q$--linear} set of $\Lambda$ of rank $t$ if it is defined by
the non-zero vectors of an $\F_q$-vector subspace $U$ of $V$ of
dimension $t$, i.e.
$$L=L_U=\{\langle {\bf u}\rangle_{\F_{q^t}}: {\bf u}\in
U\setminus\{{\bf 0}\}\}.$$
We point out that different vector subspaces can define the same linear set. For this reason a linear set and the vector space defining it must be considered as coming in pair.

Let  $\Omega=PG(W,\F_{q^t})$ be a subspace of $\Lambda$ and let $L_U$
be an $\F_q$-linear set of $\Lambda$. Then $\Omega\cap L_U$ is an
$\F_q$--linear set of $\Omega$ defined by the $\F_q$--vector
subspace $U\cap W$ and, if $dim_{\F_q}(W\cap U)=i$, we say that
$\Omega$ has {\it weight $i$} in $L_U$. Hence a point of $\Lambda$ belongs to $L_U$ if and only if it has weight at least 1 and if $L_U$ has rank $k$, then $|L_U|\leq  q^{k-1}+q^{k-2}+\dots+q+1$. For further details on linear sets see \cite{Polverino2010}, \cite{LaVa2010}, \cite{LV}, \cite{LuMaPoTr-Sub}, \cite{LP}.

An $\F_q$--linear set $L_U$ of $\Lambda$ of rank $k$ is {\em
scattered} if all of its points have weight 1, or equivalently, if
$L_U$ has maximum size $q^{k-1}+q^{k-2}+\cdots+q+1$. A scattered $\F_q$--linear set of $\Lambda$ of highest possible
rank is a {\it maximum scattered $\F_q$--linear set} of $\Lambda$; see \cite{BL2000}.

In \cite{BL2000} the authors obtain the following result on the rank of a maximum scattered linear set; see also \cite{LPhdThesis}.

\begin{theorem}\label{thm:BL} {\rm (\!\!\cite[Thms 2.1, 4.3 and 4.2]{BL2000})}
If $L_U$ is a maximum scattered $\F_q$-linear set of $PG(r-1,q^t)$
of rank $k$, then

$$k=\frac{rt}{2} \,\,\,\,\,\,\, \mbox{if} \,\, r \,  \mbox{is even},$$
$$\frac{rt-t}{2}\leq k\leq \frac{rt}{2} \,\,\,\,\,\,\, \mbox{if} \,\, r \,\,  \mbox{is odd}.$$
Also, if $rt$ is even and  $L_U$ is a maximum scattered
$\F_q$-linear set of $PG(r-1,q^t)$ of rank $\frac{rt}{2}$, then
$L_U$ is a two-intersection set (with respect to hyperplanes)
in $PG(r-1,q^t)$ with intersection numbers
$\theta_{\frac{rt}2-t-1}(q)=\frac{q^{\frac{rt}2-t}-1}{q-1}$ and $\theta_{\frac{rt}2-t}(q)=\frac{q^{\frac{rt}2-t+1}}{q-1}$.
\end{theorem}

When $r$ is even there always exists an $\F_q$--scattered linear set of rank
$\frac{rt}2$ in $PG(r-1,q^t)$ (see \cite[Theorem 2.5.5]{LPhdThesis} for an explicit
example) whereas, when $r$ is {\bf odd}, the upper bound $\frac{rt}2$ is attained in the following cases:
\begin{itemize}
\item $r=3$, $t=2$ (Baer subplanes),
\item $r=3$, $t=4$ (\!\!\cite[Section 3]{BBL2000}),
\item $r> 3$, $t=2$, $q=2$ (\!\!\cite[Thm. 4.4]{BL2000}),
\item $r\geq 3$, $(t-1)|r$ ($t$ even), $q>2$ (\!\!\cite[Thm. 4.4]{BL2000}).
\end{itemize}
This means that, for a given value of $r$, examples of maximum scattered linear sets have been shown to exist only for a small number of $t$'s.
It should be also noted that, differently from what happens for
$r$ even, in the case $r$ odd the proof of Theorem 4.4 in \cite{BL2000} shows the
existence of such maximum scattered linear sets without giving
explicit examples.

In the first part of this paper we construct three different families of scattered
$\F_q$--linear sets in $PG(2,q^t)$, $t\geq 4$ even, of rank $\frac{3t}{2}$, for infinite values of the prime power $q$. This allows us to produce for each integer $r\geq 5$, scattered
$\F_q$--linear sets in $PG(r-1,q^{t})$ of rank $\frac{rt}2$ ($t$ even). More precisely we show that

\begin{theorem}\label{mainthm}
There exist examples of scattered $\F_q$--linear sets in $PG(r-1,q^{t})$, $t$ even, of rank $\frac{rt}2$ in the following cases:
\begin{itemize}
\item $q=2$ and $t\geq 4$;
\item $q\geq 2$ and $t\not\equiv  0 \bmod 3$;
\item $q\equiv 1 \bmod 3$ and $t\equiv 0 \bmod 3$.
\end{itemize}
\end{theorem}

In the second part of the paper we point out the relationship between maximum scattered linear sets and \textcolor{black}{complete} caps in affine spaces over finite fields of even characteristic. \textcolor{black}{A cap in an affine or projective Galois space is a set of points no three of which collinear; a cap which is maximal with respect to set-theoretical inclusion is said to be complete}.
A \textcolor{black}{long-standing} issue in Finite Geometry is to ask for explicit constructions of small complete  caps in Galois spaces. The trivial lower bound for the size of a complete cap in \textcolor{black}{a Galois space of dimension $n$ and order $q$} is
\begin{equation}\label{TrivialLB}
\textcolor{black}{\sqrt 2  \cdot \sqrt q^{n-1}}.
\end{equation}
If $q$ is even and $n$ is odd, such bound is substantially sharp: \textcolor{black}{the existence of a complete cap of size
$3q+2$ in $PG(3,q)$ was showed by Segre \cite{Segre}, whose construction was later generalized by Pambianco and Storme \cite{Leo} to complete caps of size
$ 2q^{s}$ in $AG(2s+1,q)$}.
%
%
%
Otherwise, all known infinite families of complete caps have size far from \eqref{TrivialLB}; see the survey paper \cite{GiuliettiBCC}.
\textcolor{black}{Here we prove that \eqref{TrivialLB} is essentially sharp also when $n\ge 4$ is even, provided that $q$ is an even square}.
\begin{theorem}\label{MainCaps}
Let $q=2^t$, $t$ even, and $n\geq 4$ even. Then there exists a complete cap in $AG(n,q)$ of size $2\sqrt q^{n-1}.$
\end{theorem}
\textcolor{black}{Theorem \ref{MainCaps} relies on the fact that
$\F_2$--linear sets in $PG(r-1,2^t)$ of maximal rank $\frac{rt}2$, for $t$ even and $r$ odd,
 naturally correspond to complete caps in $AG(r,2^t)$ fixed by a translation group of size $2^{\frac{rt}2}$. Then the
  scattered $\F_2$--linear sets of maximal rank described in this paper, together with
the doubling construction for translation caps as described in \cite{Giulietti2007}, provide complete caps in
$AG(r+1,q)$ of size $2q^{\frac{r}{2}}$, for $q=2^t$.} \textcolor{black}{We point out that complete caps in the projective space $PG(r+1,q)$ with size of the same order of magnitude can also be constructed (see Remark \ref{completamento}).}

\section{Constructions of maximum scattered linear sets in $PG(2,q^{2n})$}
In this section we want to construct infinite families of
scattered $\F_q$--linear sets of rank $3n$ in the projective
plane $PG(2,q^{2n})$, with $n\geq 2$. Note that, by Theorem \ref{thm:BL}, such scattered linear sets are two intersection sets (with respect to the lines) of the plane.

Consider the finite field $\F_{q^{6n}}$ as a 3--dimensional vector
space over its subfield $\F_{q^{2n}}$, $n\geq 2$, and let
$\P=PG(\F_{q^{6n}},\F_{q^{2n}})=PG(2,q^{2n})$ be the associated projective plane.

The following proposition can be easily verified.

\begin{prop}\label{prop:1}
Let $f:\F_{q^{3n}}\rightarrow\F_{q^{3n}}$ be an $\F_q$--linear map, $\omega$ an element of $\F_{q^{2n}}\setminus \F_{q^n}$ and consider the
subset of $\F_{q^{6n}}$
$$U=\{f(x)+x\omega:\
x\in\F_{q^{3n}}\}$$
Then, the set \begin{equation}\label{form:scatt}
L_U=\{\langle f(x)+x\omega\rangle_{\F_{q^{2n}}}:\
x\in\F_{q^{3n}}^*\}\end{equation} is an $\F_q$--linear of rank $3n$ of the
projective plane $\P=PG(2,q^{2n})$. Also, put $$Q_f:=\Bigg\{\frac{f(x)+x\omega}{f(y)+y\omega}:\ x,y\in\F_{q^{3n}},\, y\ne 0\Bigg\},$$ the set $L_U$ turns out to be scattered
 if and only if $Q_f\cap\F_{q^{2n}}=\F_q$.
\end{prop}
\begin{proof} We first observe that $\{1,\omega\}$ is an $\F_{q^n}$--basis of $\F_{q^{2n}}$ and an
$\F_{q^{3n}}$--basis of $\F_{q^{6n}}$, as well. Also, since $f$ is an $\F_q$--linear map,
the subset $U=\{f(x)+x\omega:\ x\in\F_{q^{3n}}\}$ of $\F_{q^{6n}}$ is closed under
addition and $\F_q$--scalar multiplication, and hence it is an $\F_q$--vector subspace of $\F_{q^{6n}}$. This means that the set $L_U$
turns out to be an $\F_q$--linear set of rank $3n$ of the
plane $\P$. Also, $L_U$ is not scattered if and only if there exists a point $P_x:=\langle f(x)+x\omega\rangle_{\F_{q^{2n}}}$ of $L_U$, with $x\in\F_{q^{3n}}^*$, having weight grater
than 1, and hence there exist $y\in\F_{q^{3n}}^*$ and $\lambda\in\F_{q^{2n}}\setminus\F_q$
such that
\begin{equation}\label{form:peso}
f(x)+x\omega=\lambda(f(y)+y\omega).
\end{equation} The assertion follows.
\end{proof}

\medskip

\noindent Let now
\begin{equation}\label{form:w2}
\omega^2=A+B\omega,\end{equation} with $A,B\in\F_{q^n}$ and $A\ne
0$, and suppose that there exist $x,y\in\F_{q^{3n}}^*$ and $\lambda\in\F_{q^{2n}}\setminus\F_{q}$ satisfying Equation (\ref{form:peso}). Such an equation implies that
$$\Bigg(\frac{f(x)+x\omega}{f(y)+y\omega}\Bigg)^{q^{2n}}=\frac{f(x)+x\omega}{f(y)+y\omega},$$
i.e., taking (\ref{form:w2}) into account, we get
$$f(x)^{q^{2n}}f(y)+x^{q^{2n}}yA+(f(x)^{q^{2n}}y+f(y)x^{q^{2n}}+x^{q^{2n}}yB)\omega=$$
$$=f(y)^{q^{2n}}f(x)+y^{q^{2n}}xA+(f(y)^{q^{2n}}x+f(x)y^{q^{2n}}+y^{q^{2n}}xB)\omega.$$
Since $\{1,\omega\}$ is an $\F_{q^{3n}}$--basis of $\F_{q^{6n}}$,
the above equality is equivalent to
$$\left\{\begin{array}{l}
f(x)^{q^{2n}}f(y)-f(y)^{q^{2n}}f(x)=(xy^{q^{2n}}-yx^{q^{2n}})A\\
f(x)^{q^{2n}}y+f(y)x^{q^{2n}}-f(y)^{q^{2n}}x-f(x)y^{q^{2n}}=(xy^{q^{2n}}-yx^{q^{2n}})B
\end{array}\right..$$

\medskip

The previous arguments allow us to reformulate the previous proposition in the following way which will be useful in the sequel.
\begin{prop}\label{prop:2}
Let $f:\F_{q^{3n}}\rightarrow\F_{q^{3n}}$ be an $\F_q$--linear map and $\omega$ an element of $\F_{q^{2n}}\setminus \F_{q^n}$ such that $\omega^2=A+B\omega$, with $A,B\in\F_{q^n}$ and $A\ne
0$. The set $$
L_U=\{\langle f(x)+x\omega\rangle_{\F_{q^{2n}}}:\
x\in\F_{q^{3n}}^*\}$$ turns out to be a scattered $\F_q$--linear of rank $3n$ of the
projective plane $\P=PG(2,q^{2n})$ if and only if for each pair
$(x,y)\in\F_{q^{3n}}^*\times \F_{q^{3n}}^*$ satisfying the following equations
\begin{equation}\label{form:1}
f(x)^{q^{2n}}f(y)-f(y)^{q^{2n}}f(x)=(xy^{q^{2n}}-yx^{q^{2n}})A
\end{equation}
\begin{equation}\label{form:2}
f(x)^{q^{2n}}y+f(y)x^{q^{2n}}-f(y)^{q^{2n}}x-f(x)y^{q^{2n}}=(xy^{q^{2n}}-yx^{q^{2n}})B,
\end{equation}
the quotient
\begin{equation}\label{form:peso1}\lambda:=\frac{f(x)+x\omega}{f(y)+y\omega}
\end{equation} is an element of $\F_q^*$.  
\end{prop}

In the sequel we will exhibit examples of $\F_q$--linear maps of $\F_{q^{3n}}$ satisfying the previous properties. In particular, we face
with the monomial and the binomial cases.

\bigskip
\subsection*{Monomial case: $f(x):=ax^{q^i}$, {\small $a\in\F_{q^{3n}}^*$ and $1\leq i\leq 3n-1$}}

\bigskip

In such a case we first show that for any value of $q\geq 2$, under suitable assumptions on $a\in\F_{q^{3n}}^*$ and on the integers $i$ and $n$, we get a scattered
$\F_q$--linear set of the projective plane $PG(2,q^{2n})$ of rank $3n$. Denoting by $N_{q^{3n}/q^3}(\cdot)$ the norm function from $\F_{q^{3n}}$ over $\F_{q^3}$, we have the following

\begin{theorem}\label{thm:family(1)}
For any prime power $q\geq 2$ and any integer $n\not\equiv
0 \bmod 3$, the set $$L_U=\{\langle
ax^{q^i}+x\omega\rangle_{\F_{q^{2n}}}:\ x\in\F_{q^{3n}}^*\}$$
satisfying the following assumptions:
\bigskip

$\begin{array}{ll}
(i) & \mbox{$gcd(i,2n)=1$ and $gcd(i,3n)=3$}\\\\
(ii) & \mbox{$N_{q^{3n}/q^3}(a)\notin\F_q$ }
\end{array}
$
\bigskip

\noindent is a scattered $\F_q$--linear set of the
projective plane $PG(2,q^{2n})$ of rank $3n$.
\end{theorem}

\begin{proof}
By Proposition \ref{prop:2}, in order to prove the statement we have first
to determine the solutions $x,y\in\F_{q^{3n}}^*$ of Equations
(\ref{form:1}) and (\ref{form:2}), where we have chosen
$f(x)=ax^{q^i}$, with $a\in\F_{q^{3n}}^*$ and $1\leq i\leq 3n-1$ and
satisfying Conditions $(i)$ and $(ii)$. With these assumptions,
Equations (\ref{form:1}) and (\ref{form:2}) become
\begin{equation}\label{form:2.1}
a^{q^{2n}+1}(x^{q^{2n}}y-xy^{q^{2n}})^{q^i}=(xy^{q^{2n}}-yx^{q^{2n}})A
\end{equation}
and
\begin{equation}\label{form:2.2}
a^{q^{2n}}(x^{q^{2n+i}}y-xy^{q^{2n+i}})+a(x^{q^{2n}}y^{q^i}-x^{q^i}y^{q^{2n}})=(xy^{q^{2n}}-yx^{q^{2n}})B.\end{equation}
Let $s:=xy^{q^{2n}}-yx^{q^{2n}}$. By
\eqref{form:2.1}, if $s\ne0$, then $s$ turns out to be a solution in $\F_{q^{3n}}$ of
the equation
\begin{equation}\label{form:3}
z^{q^i-1}=-\frac{A}{a^{q^{2n}+1}}
\end{equation}
and, from Conditions $(i)$, Equation \eqref{form:3} has solutions if
and only if $N_{q^{3n}/q^3}\Big(-\frac{A}{a^{q^{2n}+1}}\Big)=1$,
namely
\begin{equation}\label{form:4}
(-1)^n N_{q^{3n}/q^3}(A)=(N_{q^{3n}/q^3}(a))^{q+1}\mbox{\quad\quad
if $2n\equiv 1 \bmod 3$}\end{equation} or
\begin{equation}\label{form:5}(-1)^n
N_{q^{3n}/q^3}(A)=(N_{q^{3n}/q^3}(a))^{q^2+1}\mbox{\quad\quad if
$2n\equiv -1 \bmod 3$}.\end{equation} Since $A\in\F_{q^n}^*$ and
since $n\not\equiv 0 \bmod 3$, we get $N_{q^{3n}/q^3}(A)\in\F_{q}$
and from Condition $(ii)$ it follows that both Equations
\eqref{form:4} and \eqref{form:5} cannot be satisfied. This means
that $s=0$ and hence $x=\alpha y$, for some $\alpha\in\F_{q^n}^*$.
Substituting in \eqref{form:2.2}, we get
\begin{equation}\label{form:6}
(\alpha^{q^i}-\alpha)(a^{q^{2n}}y^{q^{2n+i}+1}-ay^{q^{2n}+q^i})=0.
\end{equation}
If $\alpha^{q^i}\ne\alpha$, raising the previous equation to the
$q^n$--th power, then
$$y^{(q^n-1)(q^i-1)}=a^{1-q^n},$$ i.e.
$$(ay^{q^i-1})^{q^n-1}=1,$$ which is verified if and only if $y^{q^i-1}=\frac\beta
a$, for some $\beta\in\F_{q^n}^*$. This means that
$y\in\F_{q^{3n}}^*$ turns out to be a solution of the equation
$z^{q^i-1}=\beta/a$ and, from Conditions $(i)$, this happens if and
only if
$$N_{q^{3n}/q^3}(\beta)=N_{q^{3n}/q^3}(a).$$
Since $\beta\in\F_{q^n}$, from Conditions $(i)$ it follows
$N_{q^{3n}/q^3}(\beta)\in\F_q^*$ and taking Condition $(ii)$ into
account, we get a contradiction. Hence the element
$\alpha\in\F_{q^n}$ is such that $\alpha^{q^i}=\alpha$, and since
$gcd(i,n)=1$, we get $\alpha\in\F_q^*$. Substituting $x=\alpha y$ in
\eqref{form:peso1} we get $\lambda=\alpha\in\F_q^*$, proving the
assertion by Proposition \ref{prop:2}.
\end{proof}

\bigskip

Observe that the $3n$--dimensional $\F_q$--vector subspace $U$ of $F_{q^{6n}}$ defining the linear set $L_U$ of Theorem \ref{thm:family(1)} is also an $n$--dimensional $\F_{q^3}$--vector subspace. In particular, when $n=2$, $U$ is a 2--dimensional $\F_{q^3}$--subspace of $\F_{q^{12}}$ and hence it can be always seen as the set of zeros of a polynomial $$x^{q^6}+\alpha x^{q^3}+\beta x\in\F_{q^{12}}[x],$$ where $N_{q^{12}/q^3}(\beta)=1$ and $\alpha^{q^3+1}=\beta^{q^3}-\beta^{q^6+q^3+1}$; see \cite{Ball1999}. Hence, the examples of scattered $\F_q$--linear sets of rank 6 constructed in $PG(2,q^4)$ by \cite{BBL2000} belong to the family presented in Theorem \ref{thm:family(1)}.
\bigskip

Now, we will construct, for  $q\equiv 1 \bmod 3$, another family of scattered $\F_q$--linear sets of $PG(2,q^{2n})$ of rank $3n$ defined by an $\F_q$--vector subspace which is not an $\F_{q^3}$--subspace. Indeed,

\begin{theorem}\label{thm:family(2)}
For any prime power $q\equiv 1 \bmod 3$ and any integer $n\geq 2$,
the set $$L_U=\{\langle ax^{q^i}+x\omega\rangle_{\F_{q^{2n}}}:\
x\in\F_{q^{3n}}^*\}$$ satisfying the following assumptions
\bigskip

$
\begin{array}{ll}
(I) & \mbox{$gcd(i,2n)=gcd(i,3n)=1$,}\\\\
(II) & \mbox{$\Big(N_{q^{3n}/q}(a)\Big)^{\frac{q-1}{3}}\ne 1$}
\end{array}
$
\bigskip

\noindent is a scattered $\F_q$--linear set of the
projective plane $PG(2,q^{2n})$ of rank $3n$.
\end{theorem}

\begin{proof}
The first part of the proof is the same as in Theorem
\ref{thm:family(1)}. So, we have to determine the solutions
$x,y\in\F_{q^{3n}}^*$ of Equations (\ref{form:2.1}) and
(\ref{form:2.2}). Putting again $s:=xy^{q^{2n}}-yx^{q^{2n}}$, if
$s\ne0$, from the previous equality, $s$ turns out to be a solution
in $\F_{q^{3n}}^*$ of (\ref{form:3}) and, from Conditions $(I)$.
Equation (\ref{form:3}) has solutions if and only if
$N_{q^{3n}/q}\Big(-\frac{A}{a^{q^{2n}+1}}\Big)=1$, namely
\begin{equation}\label{form:7}
(N_{q^{3n}/q}(a))^{2}=(-1)^n (N_{q^{n}/q}(A))^3,
\end{equation} implying
\begin{equation}\label{form:8}
\Bigg(\Big(N_{q^{3n}/q}(a)\Big)^{\frac{q-1}{3}}\Bigg)^{2}=(-1)^{\frac{n(q-1)}3}.
\end{equation}
If $\frac{n(q-1)}3$ is even, Condition $(II)$ implies that $q$ is
odd and $\Big(N_{q^{3n}/q}(a)\Big)^{\frac{q-1}{3}}=-1$, and raising
this equality to the 3--rd power we get a contradiction. If
$\frac{n(q-1)}3$ is odd, then $q$ is even, and hence
$\Big(N_{q^{3n}/q}(a)\Big)^{\frac{q-1}{3}}=1$, again contradicting
Condition $(II)$. This means that $s=0$ and hence $x=\alpha y$, for
some $\alpha\in\F_{q^n}^*$, and arguing again as in the previous
proof, if $\alpha^{q^i}\ne\alpha$, then $y\in\F_{q^{3n}}^*$ turns
out to be a solution of the equation $z^{q^i-1}=\beta/a$, for some
$\beta\in\F_{q^n}^*$. From Conditions $(I)$, this happens if and
only if
$$(N_{q^{n}/q}(\beta))^3=N_{q^{3n}/q}(a),$$ which means that
$N_{q^{n}/q}(\beta)$ is a solution in $\F_q^*$ of the equation
$z^3=N_{q^{3n}/q}(a)$, contradicting Condition $(II)$. Hence the
element $\alpha\in\F_{q^n}^*$ is such that $\alpha^{q^i}=\alpha$, and
since $gcd(i,n)=1$, we get $\alpha\in\F_q^*$, yielding as in the previous proof
$\lambda\in\F_q^*$. By Proposition \ref{prop:2}, we have the assertion.
\end{proof}

\bigskip

Putting together Theorems \ref{thm:family(1)} and \ref{thm:family(2)} we get the following

\begin{theorem}\label{cor:1}
\begin{itemize}
\item If $n\not\equiv 0 \bmod 3$, there exist scattered $\F_q$--linear sets in $PG(2,q^{2n})$ of rank $3n$ for each prime power $q\geq
2$.
\item If $n\equiv 0 \bmod 3$, there exist scattered $\F_q$--linear sets in $PG(2,q^{2n})$ of rank $3n$ for each prime power $q\equiv 1 \bmod 3$.
\end{itemize}
\end{theorem}

\bigskip
\subsection*{Binomial case: $f(x):=ax^{q^i}+by^{q^j}$, $a,b\in\F_{q^{3n}}^*$  {\small and $1\leq i,j\leq 3n-1$}}

\bigskip

With this type of function it is clear that the linear set (\ref{form:scatt}) has $\F_q$ as maximum subfield of linearity when $\gcd(i,j,2n)=1$. In particular we will study the case when
$j=2n+i$ and, obviously $\gcd(i,2n)=1$. First of all we need a technical lemma. Denoting by $Tr_{q^{3n}/q}(\cdot)$ the trace function from $\F_{q^{3n}}$ over $\F_q$, we can consider the non--degenerate symmetric bilinear form of $\mathbb F_{q^{3n}}$
over $\mathbb F_q$ defined by the following rule $<x,y>:=Tr_{q^{3n}/q}(xy)$. Then the adjoint map $\bar \varphi$ of an $\F_q$--linear map
$\varphi(x)=\sum_{i=0}^{3n-1}a_i x^{q^i}$ of $\F_{q^{3n}}$ is $\bar
\varphi(x)=\sum_{i=0}^{3n-1}a_i^{q^{3n-i}} x^{q^{3n-i}}$ (see e.g. \cite[Sec. 2.2]{MP}). Now, we can prove the following

\begin{lemma}\label{lem2}
Let $\varphi$ be an $\mathbb F_q$-linear map of $\mathbb F_{q^{3n}}$
and $\bar \varphi$ the adjoint of $\varphi $ with respect to the bilinear form
$<,>$. Then the maps defined by $\varphi(x)/x$ and $\bar \varphi(x)/x$  have
the same image.
\end{lemma}
\begin{proof}
 Let  $\V=\mathbb F_{q^{3n}} \times \mathbb F_{q^{3n}}$
and let $\sigma: \V \times \V \longrightarrow \mathbb
F_{q^{3n}}$ be the  non-degenerate alternating bilinear  form of
$\V$ defined by $\sigma((x,y), (u,v))=xv-yu$. Then

$$\sigma'((x,y), (u,v))=Tr_{q^{3n}/q}(\sigma((x,y), (u,v)))$$
is a non-degenerate alternating bilinear form on $\V$, when $\V$ is regarded as a $6n$-dimensional vector space over $\F_{q}$. Let $\perp$ and $\perp'$ be the orthogonal
complement maps defined by $\sigma$ and $\sigma'$ on the lattices of
the $\F_{q^{3n}}$-subspaces and the $\F_{q}$-subspaces of $\V$, respectively.  Recall that  if $W$ is an $\F_{q^{3n}}$-subspace of $\V$ and $U$ is an $\F_{q}$-subspace of $\V$ then
$$dim_{\F_{q^{3n}}}W^\perp+dim_{\F_{q^{3n}}}W= 2$$ and
$$\dim_{\F_{q}}U^{\perp '}+\dim_{\F_{q}}U= 6n.$$  Also,  it is easy to
see that $W^\perp=W^{\perp '}$ for each $\F_{q^{3n}}$-subspace $W$
of $\mathbb V$ and, since $\sigma$ is an alternating form, if $W$ is
an 1-dimensional $\F_{q^{3n}}$-subspace of $\V$, then
$W^\perp=W$. Let $U_{\varphi}=\{(x,\varphi(x))\,:\, x\in
F_{q^{3n}}\}$, where $\varphi$ is an $\mathbb F_q$-linear map of
$\mathbb F_{q^{3n}}$. Then $U_{\varphi}$ is a $3n$-dimensional
$\mathbb F_q$-subspace of $\V$ and a direct calculation shows
that $U_{\varphi}^{\perp '}=U_{\bar \varphi}$. Note that an element
$t \in \mathbb F_{q^{3n}}$ belongs to the image of the map
$\varphi(x)/x$ if and only if the point $P_t=\langle (1,t)
\rangle_{\mathbb F_{q^{3n}}}$ of $PG(\V, \F_{q^{3n}})=PG(1,q^{3n})$ belongs to the
$\mathbb F_q$-linear set $L_{U_{\varphi}}$. Since
$P_t^\perp=P_t^{\perp '}=P_t$, by using the Grassmann formula, we get
$$P_t \in L_{U_{\varphi}} \Leftrightarrow dim_{\mathbb F_q}
(U_{\varphi} \cap P_t)\geq 1 \Leftrightarrow dim_{\mathbb F_q}
(U_{\varphi}^{\perp '} \cap P_t^{\perp '})\geq 1 \Leftrightarrow
dim_{\mathbb F_q} (U_{\bar \varphi} \cap P_t)\geq 1 \Leftrightarrow
P_t \in L_{U_{\bar \varphi}},$$
 i.e., $t \in \mathbb F_{q^{3n}}$
belongs to the image of the map $\varphi(x)/x$ if and only if $t$
belongs to the image of the map $\bar \varphi(x)/x$.
\end{proof}

\bigskip

\noindent Now we can show the following result.

\begin{prop}\label{prop:3}
Let $f:=f_{i,a,b}:x\in\F_{q^{3n}}\rightarrow ax^{q^i}+bx^{q^{2n+i}}\in\F_{q^{3n}}$, with $a,b\in \mathbb F_{q^{3n}}^*$ and  $gcd(i,2n)=1$, and let $\omega$ be an element of $\F_{q^{2n}}\setminus \F_{q^n}$ such that $\omega^2=A+B\omega$, with $A,B\in\F_{q^n}$ and $A\ne
0$.  If
\begin{equation}\label{eq3}
\frac{f_{i,a,b}(x)}{x}\notin \mathbb F_{q^n}\quad \text{for each }x\in \mathbb F_{q^{3n}}^*
\end{equation}
then the set $$L_U=\{\langle f_{i,a,b}(x)+wx \rangle_{\F_{q^{2n}}} \,:\, x\in \F_{q^{3n}}^*\rangle$$
turns out to be a scattered $\F_q$--linear of rank $3n$ of the
projective plane $\P=PG(\F_{q^{6n}},\F_{q^{2n}})=PG(2,q^{2n})$.
\end{prop}
\begin{proof}
\noindent
By Proposition \ref{prop:2}, in order to prove the statement we have first
to determine the solutions $x,y\in\F_{q^{3n}}^*$ of Equations
(\ref{form:1}) and (\ref{form:2}), with $f(x)=f_{i,a,b}(x)$ fulfilling Condition (\ref{eq3}).
With this choice Equation (\ref{form:1}) becomes
$$G(s):=b^{q^{2n}+1} s^{q^{2n+i}}-b^{q^{2n}}a
s^{q^{n+i}}+a^{q^{2n}+1}
s^{q^{i}}+As=0,$$
where $s=xy^{q^{2n}}-yx^{q^{2n}}$.
By (\ref{eq3}), $f_{i,a,b}(x)\ne 0$ for each $x\in\F_{q^{3n}}^*$ and then $N_{{q^{3n}/q^{n}}}(a)\neq -N_{{q^{3n}}/{q^{n}}}(b)$. Hence $G(s)=0$ if and only if
$a^{q^n} G(s)+b^{q^{2n}} G(s)^{q^n}=0$, i.e.

\begin{equation}\label{form5}
(N_{{q^{3n}}/{q^{n}}}(a)+N_{{q^{3n}}/{q^{n}}}(b))s^{q^i}+Ab^{q^{2n}}s^{q^n}+a^{q^n}As= 0.
\end{equation}

Let $L:=N_{q^{3n}/{q^{n}}}(a)+N_{q^{3n}/{q^{n}}}(b)$ and note that by (\ref{eq3})  $L\neq 0$.
This means that  if $s_0$ is a non-zero solution of (\ref{form5}), then $s_0$ satisfies the following equation
$$ \frac{b^{q^{2n-i}}s_0^{q^{n-i}}+a^{q^{n-i}}s_0^{q^{3n-i}}}{s_0}=\left(\frac{-L}{A} \right)^{q^{3n-i}},$$
i.e. there exists $s_0\in \F_{q^{3n}}^*$ such that
$$\frac{f_{n-i,b^{q^{2n-i}},a^{q^{n-i}}}(s_0)}{s_0} \in \F_{q^n},$$
and hence
$$\left(\frac{f_{n-i,b^{q^{2n-i}},a^{q^{n-i}}}(s_0)}{s_0} \right)^{q^{2n}}= \frac{f_{n-i,b^{q^{n-i}},a^{q^{3n-i}}}(s_0^{q^{2n}})}{s_0^{q^{2n}}} \in \F_{q^n}.$$
\noindent
Now, by Lemma \ref{lem2} the maps $f_{i,a,b}(x)/x$ and $\bar f_{i,a,b}(x)/x$  have the same image and a direct calculation shows that
$$\bar f_{i,a,b}=f_{n-i,b^{q^{n-i}},a^{q^{3n-i}}},$$
hence by (\ref{eq3})
$$\frac{f_{n-i,b^{q^{n-i}},a^{q^{3n-i}}}(x)}{x} \not\in \F_{q^n} \quad \text{for each }x\in \mathbb F_{q^{3n}}^*.$$
This means that Equation (\ref{form5}) only admits the zero solution, i.e.  $s=xy^{q^{2n}}-yx^{q^{2n}}=0$, which implies $x=\alpha y$ for some $\alpha\in \F_{q^n}^*$.
Now, from Equation (\ref{form:2}), substituting $f_{i,a,b}(x)=ax^{q^i}+bx^{q^{2n+i}}$ and $x=\alpha y$, we get
$$(\alpha^{q^i}-\alpha)(f_{i,a,b}(y)^{q^{2n}}y-f_{i,a,b}(y)y^{q^{2n}})=0.$$

\noindent If $\alpha^{q^i}\neq \alpha$, we get from the previous equation
$$f_{i,a,b}(y)^{q^{2n}}y=f_{i,a,b}(y)y^{q^{2n}}$$ for some $y\in\F_{q^{3n}}^*$, i.e. $f_{i,a,b}(y)/y \in \F_{q^n}$, a contradiction. Hence the
element $\alpha\in\F_{q^n}^*$ is such that $\alpha^{q^i}=\alpha$, and
since $gcd(i,n)=1$, we get $\alpha\in\F_q^*$. As in Theorems \ref{thm:family(1)} and \ref{thm:family(2)}, putting $x=\alpha y$, with $x,y\in\F_{q^{3n}}^*$ and $\alpha\in\F_{q}^*$ in
(\ref{form:peso1}) we get $\lambda=\alpha\in\F_q^*$, proving the
assertion by Proposition \ref{prop:2}.
\end{proof}

\bigskip

In what follows we will prove that if $q=2$, $i=1$ and $a=1$, then there exists at least an element $b\in \F_{q^{3n}}^*$ such that Condition (\ref{eq3}) is satisfied.
To this end we need the following preliminary result.

\begin{lemma}\label{lem:3}
 Set $n>1$ and $q=2$, consider  the function $H(t)=(1-t)/t^j$ defined in $\F_{q^{3n}}^*$, where $j=\frac{q^{2n+1}-1}{q-1}=2^{2n+1}-1$. Then there exists at least an element $b\in \mathbb{F}_{2^{3n}}^*$ not in the image of $H$ and such that $N_{{2^{3n}}/{2^{n}}}(b)\neq 1$.
\end{lemma}
\begin{proof}
 First of all notice that $\{H(t) \mid t \in \mathbb{F}_{2^{3n}}^*\}=\{ t^m+t^{m-1} \mid t \in \mathbb{F}_{2^{3n}}^*\}$, with $m= 2^{3n}-j=2^{3n}-2^{2n+1}+1$.
\noindent
The function $\theta : \mathbb{F}_{2^{3n}} \to \mathbb{F}_{2^{3n}}$, defined by $\theta(x)=x^{q^{n-1}}$, is an automorphism of $\mathbb{F}_{2^{3n}}$ and hence $$N_{{2^{3n}}/{2^n}}(x)=1 \iff N_{{2^{3n}}/{2^n}}(\theta(x))= 1.$$
Therefore
$$\exists\ b\in \mathbb{F}_{2^{3n}}^* \ : \ b \notin Im(H), \ N_{{2^{3n}}/{2^n}}(b)\neq 1 \iff$$
$$\exists\ b\in \mathbb{F}_{2^{3n}}^*\ : \ b \notin Im(\theta \circ H), \ N_{{2^{3n}}/{2^n}}(b)\neq 1.$$
\noindent
We have that
$$G(t)=(\theta \circ H)(t)=t^{(2^{3n}-2^{2n+1}+1)2^{n-1}}+t^{(2^{3n}-2^{2n+1})2^{n-1}}=t^{2^n-1}+t^{2^{n-1}-1}.$$
Since $n>1$ and $G(0)=G(1)=0$, $G(t)$ is not a permutation polynomial and it has degree $2^n-1$. Then by \cite{Turnwald1995}, its value set has size at most $2^{3n}-\frac{2^{3n}-1}{2^n-1}=2^{3n}-(2^{2n}+2^n+1)$. The number of elements of $\mathbb{F}_{2^{3n}}$ having norm over $\mathbb{F}_{2^n}$ equal to $1$ is exactly $2^{2n}+2^n+1$.
\noindent
In the following we will prove that there exist at least $2^n+2$ elements in the value set of $G$ having norm over $\mathbb{F}_{2^n}$ equal to $1$.  Note that an element having norm equal to $1$ has the form $x^{2^{n}-1}$ for some $x \in \mathbb{F}_{2^{3n}}^*$.  Consider the curve $\mathcal{C}$ defined by
$$f(x,y)=y^{2^n-1}+y^{2^{n-1}-1}+x^{2^{n}-1}=0.$$
An affine $\mathbb{F}_{2^{3n}}$-rational point of $\mathcal{C}$ having $x,y \neq 0$ corresponds to an element $b=x^{2^{n}-1}$ belonging to the image of $G$ such that $N_{{2^{3n}}/{2^n}}(b)=1$.
Intersecting the curve $\cal C$ with the lines $\ell_t\,:\, x=ty$ we get that the coordinates of the $\mathbb{F}_{2^{3n}}$-rational points of $\mathcal{C}$ having $x,y \neq 0$ are of the form
$$x=\frac{t}{(t^{2^n-1}-1)^{2^{2n+1}}} \,\,\,\,\,\,\,\, y=\frac{1}{(t^{2^n-1}-1)^{2^{2n+1}}},$$
where $t\in  \mathbb{F}_{2^{3n}} \setminus \F_{2^n}$.
Hence  $\mathcal{C}$ has exactly
$$2^{3n}-2^{n}$$
affine $\mathbb{F}_{2^{3n}}$-rational points not lying on the two axes.
Now, since the same value of $x_0^{2^n-1}$ is obtained $2^n-1$ times and since on the vertical line $x=x_0$ the curve $\mathcal C$ has at most $2^n-1$ points, we have that the same element in the image of $G$, with norm 1, can be obtained from at most $(2^n-1)^2$ points of $\mathcal C$. Then there are at least
$$\frac{2^{3n}-2^{n}}{(2^n-1)^2}>2^n+2$$
elements in the image of $G$ having norm equal to $1$. This proves that there exists at least an element $b\in \mathbb{F}_{2^{3n}}^*$ not in the image of $H$ and of norm different from $1$.
\end{proof}

Now, we are able to prove

\begin{prop}\label{prop:4}
Let $f_{i,a,b}$ be the $\F_q$--linear map of $\F_{q^{3n}}$ as defined in Proposition \ref{prop:3} and  put $i=a=1$. If $q=2$, there exists at least one element $b\in \mathbb{F}_{2^{3n}}^*$ such that
\begin{equation}\label{eq4}
\frac{f_{1,1,b}(x)}{x}\notin \mathbb F_{2^n}\quad \text{for each }x\in \mathbb F_{2^{3n}}^*.
\end{equation}
\end{prop}
\begin{proof}
Taking $q=2$ and $i=a=1$ in $f_{i,a,b}(x)=ax^{2^i}+bx^{2^{2n+i}}$, Condition \eqref{eq4} reads
\begin{equation}\label{eq5}\frac{x^2+bx^{2^{2n+1}}}{x}=x+bx^{2^{2n+1}-1}\notin \mathbb F_{2^n}\quad \text{for each }x\in \mathbb F_{2^{3n}}^*.
\end{equation}
Let $g(x):=\frac{f_{1,1,b}(x)}{x}=x+bx^{2^{2n+1}-1}$ for each $x\in \mathbb F_{2^{3n}}^*$. Note that since $g(\eta x)=\eta g(x)$ for each $\eta\in\F_{q^{n}}$, Condition (\ref{eq5}) is satisfied if $g(x)\ne 0$ and $g(x)\ne 1$ for each $x\in \mathbb F_{2^{3n}}^*$. If there is an element $x_0\in\F_{q^{3n}}^*$ such that $g(x_0)=1$, then the corresponding $b$ belongs to the image of the function $H$ defined in Lemma \ref{lem:3}. If there is an element $x_0\in\F_{q^{3n}}^*$ such that $g(x_0)=0$, then the corresponding $b$ has norm equal to 1. By Lemma \ref{lem:3} there is an element $b_0\in \mathbb{F}_{2^{3n}}^*$ not belonging to the image of $H$ and having norm different from $1$. This implies that  Condition \eqref{eq5}, for $b_0$, is satisfied and hence $f_{1,1,b_0}$ satisfies Condition \eqref{eq4}.
\end{proof}

Putting together Propositions \ref{prop:3} and \ref{prop:4} we get the following
\begin{theorem}\label{thm:family(3)}
For each integer  $n>1$,
the set $$L_U=\{\langle x^{2}+bx^{2^{2n+1}}+x\omega\rangle_{\F_{2^{2n}}}:\
x\in\F_{2^{3n}}^*\},$$ where $b\in\F_{2^{3n}}^*$ with $N_{2^{3n}/{2^n}}(b)\ne 1$ and such that
$$
x+bx^{2^{2n+1}-1}\notin \mathbb F_{2^n}\quad \text{for each }x\in \mathbb F_{2^{3n}}^*
$$
is a scattered $\F_2$--linear set of the
projective plane $PG(2,2^{2n})$ of rank $3n$.
\end{theorem}

\begin{remark}
{\rm
MAGMA computational results show that for $n=3$ and $q\in\{3,4,5\}$ there exist elements $b\in\F_{q^{3n}}^*$ for which the functions $f_{1,1,b}$ satisfy Condition \eqref{eq4} yielding $\F_q$--scattered linear sets in $PG(2,q^6)$, $q\in\{3,4,5\}$, of rank 9. However, taking Theorems \ref{cor:1} and \ref{thm:family(3)} into account, the existence of a family of scattered $\F_q$--linear sets in $PG(2,q^{2n})$ for each $n\equiv 0 \bmod 3$, $q\not\equiv 1 \bmod 3$ and $q>2$, remains an open problem. }
\end{remark}

\section{Constructions in $PG(r-1,q^{t})$}
First of all we prove the following
\begin{theorem}\label{prop:generaliz}
Let $\P=PG(\V,F_{q^t})=PG(r-1,q^t)$ be a projective space and let
\begin{equation}\label{eq:sommadir}
\V=V_1\oplus_{_{\F_{q^{t}}}}\dots\oplus_{_{\F_{q^{t}}}} V_{m},
\end{equation}
with $dim\,V_i=s_i\geq 2$ and $i\in\{1,\dots,m\}$. If $L_{U_i}$ is a scattered $\F_q$--linear set of $PG(V_i,\F_{q^t})=PG(s_i-1,q^t)$ then $L_W$, where
\begin{equation}\label{eq:sommadir1}
W=U_1\oplus_{_{\F_{q}}}\dots\oplus_{_{\F_{q}}}U_{m},\end{equation}
is a scattered $\F_q$--linear set of $\P$.

Also, $L_W$ has maximum rank $\frac{rt}2$ if and only if each $L_{U_i}$ has maximum rank $\frac{s_it}{2}$.
\end{theorem}

\begin{proof}
Let $k_i$ be the rank of $L_{U_i}$. By Theorem \ref{thm:BL} $k_i\leq\frac{s_it}{2}$ for each  $i\in\{1,\dots,m\}$. It is clear that $L_W$ is an $\F_q$--linear set of $\P$ of rank $\sum_{i=1}^mk_i\leq \sum_{i=1}^m\frac{s_it}{2}=\frac{rt}2$. If $P:=\langle\underline{w}\rangle$ is a point of $L_W$ with weight grater than 1, then there exist $\underline{w}'\in W$, $\underline{w}'\ne\underline{0}$, and $\lambda\in\F_{q^{t}}\setminus\F_q$ such that $\underline{w}=\lambda \underline{w}'$. By (\ref{eq:sommadir1}), the vectors $\underline{w}$ and $\underline{w}'$ can be uniquely written as
$$\underline{w}=\underline{u}_1+\dots+\underline{u}_{m}\mbox{\quad and \quad}\underline{w}'=\underline{u}'_1+\dots+\underline{u}'_{m},$$
where $\underline{u}_i,\underline{u}'_i\in U_i$ for each $i\in\{1,\dots,m\}$. Taking $\underline{w}=\lambda \underline{w}'$ and (\ref{eq:sommadir}) into account,  from the previous equalities we get $\underline{u}_i=\lambda \underline{u}'_i$ for each $i\in\{1,\dots,m\}$. Suppose that $j\in\{1,\dots,m\}$ is the smallest number such that $\underline{u}_j\ne\underline{0}$. Then $\underline{u}_j=\lambda \underline{u}'_j$, with $\underline{u}_j\in U_j$, and since $L_{U_j}$ is scattered we get $\lambda\in\F_q^*$, a contradiction. The last part is obvious.
\end{proof}

The previous theorem can be naturally applied when $r$ is even by considering scattered $\F_q$--linear sets of rank $t$ on $\frac r 2$ lines, say $\ell_i$, spanning the whole space $PG(r-1,q^t)$. In such a way we get a scattered $\F_q$--linear set in $PG(r-1,q^t)$ of rank $\frac{rt}2$. We will call this construction of {\it type $(C1)$}. Some scattered linear sets reflecting this construction are those called {\it of pseudoregulus type} (see \cite[Definitions 3.1 and 4.1]{LuMaPoTr-Sub}), for which each scattered linear set on $\ell_i$ is of pseudoregulus type (see \cite[Remark 4.5]{LuMaPoTr-Sub}). Linear sets of pseudoregulus type have been also studied in \cite{MPT2007}, \cite{LMPT}, \cite{LV} and to this family belongs the first explicit example of scattered linear sets obtained by Construction $(C1)$ (see proof of \cite[Thm. 2.5.5]{LPhdThesis}). Also, from \cite[Example 4.6 $(i)$ and $(ii)$]{LuMaPoTr-Sub} it is clear that, by using Construction $(C1)$, we can also obtain scattered linear sets in $PG(r-1,q^t)$, $r$ even, of rank $\frac{rt}2$ which are not of pseudoregulus type.

\section*{Proof of Theorem \ref{mainthm}}

Putting together Theorems \ref{thm:family(1)}, \ref{thm:family(2)} and \ref{thm:family(3)}  and Theorem \ref{prop:generaliz}, it follows that when $t$ is even and $r\geq 5$ we have several ways to construct scattered $\F_q$--linear sets in $PG(\V,F_{q^t})=PG(r-1,q^t)$ of rank $\frac{rt}{2}$, by decomposing $\V$ as a direct sum over $\F_{q^t}$ of vector spaces of dimension 2 and 3, proving in this way Theorem \ref{mainthm}. Obviously, the greater is the integer $r$, the wider are these possible constructions.

\begin{remark}
{\rm From Theorem \ref{thm:BL}, each scattered $\F_q$-linear set of $PG(r-1,q^{2n})$ of rank $rn$ is a two--intersection set of the space with respect to the hyperplanes with intersection numbers
$\theta_{(r-2)n-1}(q)=\frac{q^{(r-2)n}-1}{q-1}$ and $\theta_{(r-2)n}(q)=\frac{q^{(r-2)n+1}-1}{q-1}$. Then, $L_U$ is a {\it $\theta_{(r-2)n-1}(q)$--fold blocking set} (with respect to hyperplanes) in $PG(r-1,q^{2n})$ (\cite[Thm. 6.1]{BL2000}) and gives to rise two-weight linear codes and strongly regular graphs (see \cite{CK} and \cite[Sec. 5]{BL2000}). As observed in \cite{BL}, we want to stress that the parameters of these two--intersection sets are not new. Indeed, sets with the same parameters can be obtained by taking the disjoint union of $\frac{q^n-1}{q-1}$ Baer subgeometries in $PG(r-1,q^{2n})$ isomorphic to $PG(r-1,q^{n})$. This set is called {\it of type $I$} in \cite{BL}. Also in \cite[Thm. 2.2]{BL}, the authors show that a scattered $\F_q$--linear set of maximum rank cannot contain any Baer subgeometry of $PG(r-1,q^{2n})$ and hence the corresponding two--intersection set is not isomorphic to a set of type $I$.
}
\end{remark}

\section{Small complete caps from maximum scattered linear sets}
\textcolor{black}{Many links between the theory of linear sets and a large number of geometrical objects are known. Among them,  two-intersection sets, blocking sets or multiple blocking sets, translation ovoids of polar spaces, translation spreads of the Cayley Generalized Hexagon $H(q)$.  Also, linear sets are widely used in the construction of finite semifields. In this section we describe a connection between $F_2$--linear sets and another classical object in Finite Geometry: complete caps in Galois spaces. Such a connection is indeed fruitful; in fact, the results of the previous section on $F_2$--linear sets provide a solution, for spaces of even square order, to the long-standing problem of establishing whether the theoretical lower bound for the size of a complete cap is substantially sharp.}

\textcolor{black}{We first recall a Definition from \cite[Sec. 2]{Giulietti2007}.}
\begin{defin}
\textcolor{black}{Let $q=2^t$ and let $G$ be an additive subgroup of $\mathbb{F}_q^r$. Let}
$$\cK_{G}:=\{P_v\; |\; v \in G\} \subset AG(r,q),$$
where $P_v$ is the affine point \textcolor{black}{with} coordinates $(a_1,a_2,\ldots,a_r)$ corresponding to the vector $v=(a_1,a_2,\ldots,a_r) \in \mathbb{F}_q^r$. A \emph{translation cap} is a cap in $AG(r,q)$ which coincides with $\cK_G$ for some additive subgroup $G$ of $\mathbb{F}_q^r$.
\end{defin}

Translation caps can be characterized as follows.
\begin{theorem}\label{Th:TranslCap}{\rm \cite[Lemma 2.1]{Giulietti2007}}
For an additive subgroup $G$ of $\mathbb{F}_q^r$, $q$ even, the set $K_G$ is a translation cap if and only if any two non-zero distinct vectors in $G$ are $\mathbb{F}_q$-linearly independent.
\end{theorem}

\begin{prop}\label{Theorem:correspondance}
An $\F_2$--scattered linear set in $PG(r-1,2^t)$, $t>1$, corresponds to a translation cap in $AG(r,2^t)$ and viceversa.
\end{prop}
\begin{proof}
Let $U$ be an $\mathbb{F}_2$-vector subspace of $V=\mathbb{F}_{2^t}^{r}$, $t>1$, corresponding to the scattered linear set $L_U$ in $PG(V,\F_{2^t})$. Since $U$ in an additive subgroup of $V$, by Theorem \ref{Th:TranslCap}, $\mathcal{K}_U$ is a translation cap if and only if there no two distinct vectors in $U$ are $\mathbb{F}_{2^t}$-linearly dependent. This happens if and only all the elements of $U$ correspond to distinct points of $L_U$, that is $L_U$ is a scattered $\F_2$--linear set.
\end{proof}

Let $\mathcal{SL}$ and $\mathcal{TC}$ be the sets of all the scattered linear sets in $PG(r-1,2^t)$ and all the translation caps in $AG(r,2^t)$. Form the previous theorem we can deduce the existence of a bijective function
$$\varphi : \mathcal{SL} \to \mathcal{TC}$$
which sends $L_U$ to $\varphi(L_U)=\mathcal{K}_U$ for each $\mathbb{F}_2$-vector subspace $U$ of $V=\mathbb{F}_{2^t}^r$.

\begin{prop}
Let $U_1$ and $U_2$ such that $\mathcal{K}_{U_1}$ and $\mathcal{K}_{U_2}$ are equivalent under the action of $A\Gamma L(r,2^t)$. Then $\varphi^{-1}(\mathcal{K}_{U_1})$ and $\varphi^{-1}(\mathcal{K}_{U_2})$ are equivalent under the action of $P\Gamma L(r,2^t)$.
\end{prop}
\begin{proof}
Let $f\in A\Gamma L(r,2^t)$ be such that $f(\mathcal{K}_{U_1})=\mathcal{K}_{U_2}$. Then $f$ contains no translations, since it has to fix the $0$-vector. Then $f=M\tau$ with $M \in GL(r,2^t)$ and $\tau \in Aut (\mathbb{F}_{2^t})$. Let $$\mathcal{K}_{U_1}=\{0,P_1,P_2,\ldots,P_n\}, \qquad \mathcal{K}_{U_2}=\{0,Q_1,Q_2,\ldots,Q_n\},$$ with $f(0)=0$ and $f(P_i)=Q_i$. Also, let $\varphi^{-1}(\mathcal{K}_{U_1})=\{\widetilde{P}_1,\widetilde{P}_2,\ldots,\widetilde{P}_n\}$ and $\varphi^{-1}(\mathcal{K}_{U_2})=\{\widetilde{Q}_1,\widetilde{Q}_2,\ldots,\widetilde{Q}_n\}$, with $\widetilde{P}_i=\lambda_i P_i$,  $\widetilde{Q}_i=\mu_i Q_i$ and $\lambda_i, \mu_i \neq 0$ for all $i=1,\ldots,n$. Consider $g=\frac{1}{\det(M)}M\tau \in P\Gamma L(r,2^h)$. Then
$$g(\widetilde{P}_i)=\left(\frac{1}{\det(M)}M\tau\right) (\widetilde{P}_i)=\left(\frac{1}{\det(M)}M\tau\right) (\lambda_i P_i)=$$
$$=\frac{\tau(\lambda_i)}{\det(M)}\left(M\tau\right) (P_i)=\frac{\tau(\lambda_i)}{\det(M)}Q_i=\frac{\tau(\lambda_i)}{\mu_i\det(M)}\widetilde{Q}_i.$$
Then $g$ sends $\widetilde{P}_i$ to $\widetilde{Q}_i$ and $\varphi^{-1}(\mathcal{K}_{U_1})$ is projectively equivalent to  $\varphi^{-1}(\mathcal{K}_{U_2})$.
\end{proof}

By \cite[Proposition 2.5]{Giulietti2007}  the maximum size of   a translation \textcolor{black}{cap} in $AG(r, q)$, $q=2^t$ and $t>1$, is  $q^{\frac{r}{2}}$; if the bound is attained then \textcolor{black}{the cap is said to be a  \emph{maximal} translation cap}.  \textcolor{black}{We recall two further results from \cite{Giulietti2007}}.

\begin{lemma}{\rm \cite[Proposition 2.8]{Giulietti2007}}\label{Proposition 2.8} If $\mathcal{K}_G$ is a maximal translation cap in $AG(r, 2^t)$, and $\mathcal{K}_H$ a maximal translation cap in $AG(\overline{r}, 2^t)$, then $\mathcal{K}_G\times \mathcal{K}_H$ is a maximal translation cap in $AG(r+\overline{r}, 2^t)$.
\end{lemma}
\begin{lemma}{\rm (Doubling construction) \cite[Corollary 2.12]{Giulietti2007}}\label{Corollary 2.12} If $\mathcal{K}_G$ is a maximal translation cap in $AG(r, 2^t)$, then $\mathcal{K}_{G\times \{0,1\}}$ is a complete cap in $AG(r+1, 2^t)$.
\end{lemma}
\textcolor{black}{We are now in a position to prove the key result of this section.}
%

\begin{prop}\label{MainPropCaps} Let $q=2^t$, $t$ even, and $n\geq 4$ even. If there exists a maximum scattered linear set in $PG(2,q)$, then there exists a complete cap in \textcolor{black}{$AG(n,q)$ with size $2q^{\frac{n-1}{2}}.$}\end{prop}

\begin{proof} Let $L$ be a maximum scattered $\F_2$--linear set of $PG(2,2^t)$. Since $t$ is even  it has rank $\frac{3t}{2}$. By Proposition \ref{Theorem:correspondance} it is equivalent to a translation cap $\mathcal{K}$ in $AG(3,2^t)$ of size $2^{\frac{3t}{2}}=\sqrt q^{\,3}$. Since the upper bound of \cite[Proposition 2.5]{Giulietti2007} is attained, $\mathcal{K}$ is a maximal translation cap in $AG(3,2^t)$. Let $n\geq 4$ even and consider in $AG(n-1,2^t)$ the following cap of size $q^{\frac{n-1}{2}}$:
$$\overline{\mathcal{K}}=\left\{\left(a,b,c,x_1,x_1^2,x_2,x_2^2,\ldots,x_{\frac{n-4}{2}},x_{\frac{n-4}{2}}^2\right) : (a,b,c) \in \mathcal{K}, x_i \in \F_{2^t}\right\}.$$
By Lemma \ref{Proposition 2.8}, \textcolor{black}{together with the fact that $\{(x,x^2) : x \in \F_{2^t})\}$
is a translation cap in $AG(2,2^t)$,  $\overline{\mathcal{K}}$} is a maximal translation cap in \textcolor{black}{$AG(n-1,2^t)$}. Now the cap
$$\overline{\overline{\mathcal{K}}}=\left\{\left(a_1,\ldots,a_{n-1},0\right) : (a_1,\ldots,a_{n-1}) \in \overline{\mathcal{K}}\right\} \cup \left\{\left(a_1,\ldots,a_{n-1},1\right) : (a_1,\ldots,a_{n-1}) \in \overline{\mathcal{K}}\right\}$$
is a complete translation cap in \textcolor{black}{$AG(n,2^t)$} of size $2q^{\frac{n-1}{2}}$ by Lemma \ref{Corollary 2.12}.
\end{proof}

The existence of a complete cap in $AG(n,q)$ of size $2q^{\frac{n-1}{2}}$, for $n\ge 4$ even and $q$ an even square,
now follows from Theorems \ref{thm:family(1)}, \ref{thm:family(2)}, \ref{thm:family(3)} and Proposition \ref{MainPropCaps}.
Theorem \ref{MainCaps} in Introduction is then proved.

\begin{remark}\label{completamento}{\rm
\textcolor{black}{For $q$ an even square and $n\ge 4$ even, the trivial lower bound for complete caps is substantially sharp not only in the affine space $AG(n,q)$ but also in the projective space $PG(n,q)$.
In fact, it is possible to show that in ${PG}(2k+4,q)$, $k \geq 0$ there exists a complete cap of size at most $3q^{k+\frac{3}{2}}+4q^{k+1}+3\frac{q^{k+1}-1}{q-1}$ containing the translation cap of size $2q^{k+\frac{3}{2}}$ obtained in Theorem \ref{MainCaps}. The lengthy and technical proof is similar to those of {\rm \cite[Theorem 4.7]{Giulietti2007}} and {\rm \cite[Propositions 2.5 and 5.3]{GPIEEE}}, where a complete translation cap in $AG(n,q)$ is extended to a complete cap in $PG(n,q)$}.}

\end{remark}

\bigskip

\noindent Daniele Bartoli and Massimo Giulietti\\
Dipartimento di Matematica e Informatica,\\
Universit\`a degli Studi
di Perugia,\\
I--\,06123 Perugia, Italy\\
{\em daniele.bartoli@unipg.it}, {\em massimo.giulietti@unipg.it}
\vspace*{.3cm}

\noindent Giuseppe Marino and Olga Polverino\\
Dipartimento di Matematica e Fisica,\\
 Seconda Universit\`a degli Studi
di Napoli,\\
I--\,81100 Caserta, Italy\\
{\em giuseppe.marino@unina2.it}, {\em olga.polverino@unina2.it}

\end{document}